\documentclass[a4paper, 11pt, twoside]{article}

\usepackage{amsmath, amscd, amsfonts, amssymb, amsthm, latexsym, url, color, framed}

\usepackage{todonotes}
\usepackage{graphicx}
\usepackage[left=1.3in, right=1.2in, top=1.3in, bottom=0.9in, includefoot, headheight=13.6pt]{geometry}

\input{xy}
\xyoption{all}

\usepackage[ps2pdf=true,colorlinks,breaklinks=true]{hyperref}
\usepackage[figure,table]{hypcap}
\hypersetup{
	bookmarksnumbered,
	pdfstartview={FitH},
	citecolor={blue},
	linkcolor={red},
	urlcolor={black},
	pdfpagemode={UseOutlines}
}
\makeatletter
\newcommand\org@hypertarget{}
\let\org@hypertarget\hypertarget
\renewcommand\hypertarget[2]{%
  \Hy@raisedlink{\org@hypertarget{#1}{}}#2%
} 
\makeatother

\newtheorem{theorem}{Theorem}[section]
\newtheorem{lemma}[theorem]{Lemma}
\newtheorem{corollary}[theorem]{Corollary}
\newtheorem{proposition}[theorem]{Proposition}

\theoremstyle{definition}
\newtheorem{definition}[theorem]{Definition}
\newtheorem{remark}[theorem]{Remark}
\newtheorem{example}[theorem]{Example}

\newcommand{\xysquare}[8]{
\[\xymatrix{
#1 \ar@{#5}[r] \ar@{#6}[d] & #2 \ar@{#7}[d]\\
#3 \ar@{#8}[r] & #4
}\]
}

\DeclareMathOperator*{\holim}{\operatorname*{holim}}

\newcommand{\al}{\alpha}
\newcommand{\bb}{\mathbb}

\newcommand{\blob}{\bullet}

\newcommand{\comment}[1]{}

\newcommand{\comp}{{\hat{\phantom{o}}}}

\newcommand{\into}{\hookrightarrow}
\newcommand{\isoto}{\stackrel{\simeq}{\to}}
\newcommand{\Isoto}{\stackrel{\simeq}{\longrightarrow}}

\newcommand{\op}{\operatorname}
\newcommand{\pid}[1]{\langle #1 \rangle}
\newcommand{\quis}{\stackrel{\sim}{\to}}
\newcommand{\res}{\overline}
\newcommand{\roi}{\mathcal{O}}

\newcommand{\sub}[1]{{\mbox{\scriptsize #1}}}

\newcommand{\To}{\longrightarrow}

\newcommand{\xto}{\xrightarrow}

\newcommand{\THH}{T\!H\!H}
\newcommand{\HH}{H\!H}
\newcommand{\HC}{H\!C}
\newcommand{\TR}{T\!R}
\newcommand{\TC}{T\!C}

\renewcommand{\cal}{\mathcal}
\renewcommand{\hat}{\widehat}
\renewcommand{\frak}{\mathfrak}
\newcommand{\indlim}{\varinjlim}
\renewcommand{\tilde}{\widetilde}

\renewcommand{\projlim}{\varprojlim}

\DeclareMathOperator{\Hom}{Hom}

\DeclareMathOperator{\Spec}{Spec}

\DeclareMathOperator{\Tor}{Tor}

\newcommand{\overbar}[1]{\mkern 1.5mu\overline{\mkern-2.5mu#1\mkern-1.5mu}\mkern 1.5mu}

\usepackage{fancyhdr}

\usepackage{sectsty}
\sectionfont{\Large\sc\centering}
\chapterfont{\large\sc\centering}
\chaptertitlefont{\LARGE\centering}
\partfont{\centering}

\begin{document}
\title{Pro cdh-descent for cyclic homology and $K$-theory}
\author{Matthew Morrow}
\comment{
\footnote{University of Chicago, 5734 S.~University Ave., Chicago, IL, 60637, USA
\quad 312-810-9617 \newline {\tt mmorrow@math.uchicago.edu}\quad
\url{http://math.uchicago.edu/\~mmorrow/}\newline  Supported by a Simons Postdoctoral Fellowship.}
}
\date{}

\maketitle

\begin{abstract}
In this paper we prove that cyclic homology, topological cyclic homology, and algebraic $K$-theory satisfy a pro Mayer--Vietoris property with respect to abstract blow-up squares of varieties, in both zero and finite characteristic. This may be interpreted as the well-definedness of $K$-theory with compact support.
\end{abstract}

\thispagestyle{empty}
\tableofcontents

\setcounter{section}{-1}
\section{Introduction}

The primary goal of this paper is to prove that algebraic $K$-theory satisfies a pro Mayer--Vietoris property with respect to abstract blow-up squares of varieties, in both zero and finite characteristic, at least assuming strong resolution of singularities. This may also be interpreted as the well-definedness of a theory of $K$-groups with compact support.

To state the main result, we first recall that an {\em abstract blow-up square} of schemes is a pull-back diagram \xysquare{Y'}{X'}{Y}{X}{->}{->}{->}{->} where $X'\to X$ is proper, $Y\to X$ is a closed embedding, and the induced map on the open subschemes $X'\setminus Y'\to X\setminus Y$ is an isomorphism. We denote by $rY$ the $r^\sub{th}$ infinitesimal thickening of $Y$ inside $X$, i.e., $rY:=\Spec\roi_X/\cal I^r$ where $\cal I\subseteq\roi_X$ is the sheaf of ideals defining $Y$, and similarly for other closed embeddings.

Our pro descent theorem is the following:

\begin{theorem}[Pro cdh-descent for $K$-theory; see Thm.~\ref{theorem_pro_descent_for_K_theory}]\label{theorem6}
Let $k$ be an infinite, perfect field which has strong resolution of singularities. For any abstract blow-up square, as above, consisting of finite type $k$-schemes, the resulting square of pro spectra \xysquare{K(X)}{K(X')}{\{K(rY)\}_r}{\{K(rY')\}_r}{->}{->}{->}{->} is homotopy cartesian. In other words, the canonical maps of pro abelian relative $K$-groups \[\{K_n(X,rY)\}_r\To\{K_n(X',rY')\}_r\] is an isomorphism for all $n\in\bb Z$.
\end{theorem}

The main applications of Theorem \ref{theorem6} are to zero cycles; since the Theorem may be accepted as a block box for these applications, they are presented separately in an accompanying paper \cite{Morrow_zero_cycles}. There we solve cases of an outstanding conjecture of V.~Srinivas and S.~Bloch \cite[pg.~6]{Srinivas1985a} concerning the Levine--Weibel Chow group of zero cycles on singular varieties, and relate Chow groups with modulus, which play a prominent role in M.~Kerz and S.~Saito's higher dimensional class field theory \cite{KerzSaito2013}, to both Levine--Weibel Chow groups and $K$-theory via cycle class maps.

Theorem \ref{theorem6} may be interpreted as a definition of {\em $K$-theory with compact support}, as follows. Given a separated $k$-variety $X$, we may choose a proper compactification $\overbar X$, set $Y:=\overbar X\setminus X$, and then define $K^c(X):=\op{holim}_rK(\overbar X,rY)$. Theorem \ref{theorem6} implies that this definition does not depend on the chosen compactification $\overbar X$ of $X$, and we prove in Proposition \ref{proposition_compact2} that there is a a pushforward localisation sequence relating $K^c(U)$ and $K^c(X)$ for any open subvariety $U\subseteq X$, thereby justifying the nomenclature of $K^c$ as a theory with compact support.

In characteristic zero we actually establish Theorem \ref{theorem6} in much greater generality, namely for any abstract blow-up square of Noetherian, quasi-excellent, $\bb Q$-schemes of finite Krull dimension. This is obtained by first extending Haesemeyer's argument \cite[\S5--6]{Haesemeyer2004}, for checking whether a pre-sheaf of spectra satisfies cdh-descent, to this generality; see Proposition \ref{proposition_Haesemeyer}. This also allows us to prove the vanishing part of Weibel's $K$-dimension conjecture for such schemes; see Theorem \ref{theorem_K_dim}.

To contextualise Theorem \ref{theorem6} we should make a few comments about cdh-descent. Algebraic $K$-theory (as well as Andr\'e--Quillen, Hochschild, topological Hochschild, cyclic, and topological cyclic homologies) does not satisfy descent in V.~Voevodky's cdh-topology \cite{Voevodsky2010}, in contrast to, e.g., Weibel's homotopy invariant $K$-theory \cite{Haesemeyer2004}, or periodic cyclic homology \cite{Cortinas2008}. Since it is known that $K$-theory does satisfy Nisnevich descent, this precisely means that taking the $K$-theory of an abstract blow-up square does not necessarily yield a homotopy a cartesian square of spectra. In fact, the failure of $K$-theory to satisfy cdh-descent in characteristic zero (resp.~$p>0$) is precisely equal to the analogous obstruction in cyclic homology \cite{Cortinas2006, Cortinas2008, Haesemeyer2004} (resp.~topological cyclic homology \cite{GeisserHesselholt2010}); this has led to new calculations of the $K$-theory of singular algebraic varieties, especially by G.~Corti{\~n}as, C.~Haesemeyer, M.~Schlichting, M.~Walker, and C.~Weibel \cite{Cortinas2008, Cortinas2010, Cortinas2008a, Cortinas2009, Cortinas2013} in characteristic zero. The moral of Theorem \ref{theorem6} is that the failure of $K$-theory to satisfy cdh-descent can be remedied by taking all infinitesimal thickenings of the exceptional subschemes $Y$ and $Y'$ into account.

We now briefly discuss how Theorem \ref{theorem6} is proved and simultaneously state our other main results. Using the aforementioned cdh comparison between $K$-theory and cyclic homology, Theorem \ref{theorem6} in characteristic zero may be reduced to an analogous pro Mayer--Vietoris assertion about cyclic homology, or even Hochschild or Andr\'e--Quillen homology. We establish this pro Mayer--Vietoris property in considerable generality (see Theorem \ref{theorem_pro_cdh_for_AQ} for the Andr\'e--Quillen case):

\begin{theorem}[Pro descent for $\HH$ and $HC$ wrt.~abstract blow-ups; see Thm.~\ref{theorem_pro_cdh_for_HH}]\label{theorem2}
Let $k$ be a Noetherian ring, and let
\[\xymatrix{
Y'\ar[d]\ar[r] & X'\ar[d]\\
Y\ar[r] & X
}\]
be an abstract blow-up square of Noetherian, finite Krull dimensional $k$-schemes. Then the canonical maps \[\{\HH_n^k(X,rY)\}_r\To \{\HH_n^k(X',rY')\}_r,\quad\quad \{\HC_n^k(X,rY)\}_r\To \{\HC_n^k(X',rY')\}_r\] are isomorphisms of pro abelian groups for all $n\in\bb Z$.
\end{theorem}

We should mention at this point that our Hochschild and cyclic homology groups are always computed in the derived sense, that is after replacing rings by free simplicial resolutions (see Section \ref{subsection_HH_and_HC_of_schemes} for details).

The key input to proving Theorem \ref{theorem2} is new formal function theorems for Andr\'e--Quillen, Hochschild, and cyclic homology, in the style of Grothendieck's formal functions theorem for coherent cohomology. In particular, in the case of Hochschild homology we prove the following (see Corollary \ref{corollary_traditional_formal_functions_for_AQ} for the Andr\'e--Quillen case):

\begin{theorem}[Formal functions theorem for $\HH$; see Thm.~\ref{theorem_formal_function_for_HH}(v)]\label{theorem_intro_HH}
Let $A$ be a Noetherian ring, $I\subseteq A$ an ideal, and $X$ a proper scheme over $A$ of finite Krull dimension. Then the canonical map \[\{\HH_n^A(X)\otimes_AA/I^r\}_r\To \{\HH_n^{A/I^r}(X\times_AA/I^r)\}_r\] is an isomorphism of pro $A$-modules for all $n\in\bb Z$.
\end{theorem}

If $A$, as in the previous theorem, is moreover $I$-adically complete, then it follows that the canonical map $HH_n^A(X)\To\projlim_rHH_n^{A/I^r}(X\times_AA/I^r)$ is an isomorphism for all $n\in\bb Z$, which is the statement closest to the usual formulation of Grothendieck's formal functions theorem for coherent cohomology. 

The proof of Theorem \ref{theorem6} in finite characteristic is similar in outline, except that the results concerning Andr\'e--Quillen, Hochschild, and cyclic homology in Sections \ref{section_pro_h_for_AQ} and \ref{subsection_HH_and_HC_of_schemes} must be replaced by similar results for topological Hochschild and cyclic homology. These results are not contained in the current paper, but may rather be found in recent joint work with B.~Dundas \cite{Morrow_Dundas}, which includes in particular the $\THH$ analogue of Theorem \ref{theorem_intro_HH}. The few details missing from [op.~cit.], where for example the analogue of Theorem \ref{theorem2} for $\THH$ and $\TC$ was not explicitly stated, are given in Section \ref{subsection_HH_and_HC_of_schemes}.

\subsection*{Notation, etc.}
All rings are commutative, associative, and unital. If $A_\blob$ is a simplicial abelian group, then we tend to abuse notation by speaking of the homology groups $H_n(A_\blob)$ of $A_\blob$, rather than more correctly (but equivalently) the simplicial homotopy groups $\pi_n(A_\blob)$ or the homology groups $H_n(NA_\blob)$ of the normalised complex. This should not cause any confusion.

\subsection*{Acknowledgments}
It is a pleasure to take this opportunity to thank B.~Dundas for the collaboration \cite{Morrow_Dundas}, without which Theorem \ref{theorem6} would be restricted to characteristic zero. The majority of this work was done while I was supported by the Simons foundation as a postdoctoral fellow at the University of Chicago, and I am grateful to the foundation for their support.

\section{Review of pro modules and Artin--Rees properties}\label{section_pro_review}
\subsection{Pro abelian groups and pro modules}
Everything we need about categories of pro objects may be found in one of the standard references, such as the appendix to \cite{ArtinMazur1969}, or \cite{Isaksen2002}. We will often use $\op{Pro}(A\op{-}mod)$, the category of pro $A$-modules for some ring $A$, and $\op{Pro}Ab$, the category of pro abelian groups.

Let $\cal C$ be a category. In this paper, an object of $\op{Pro}\cal C$ is simply an inverse system $\cdots\to A_2\to A_1$ of objects and morphisms in $\cal C$, which is denoted $\{A_r\}_r$ or very occasionally $A_\infty$. Morphisms in $\op{Pro}\cal C$ are given by the rule \[\Hom_{\op{Pro}\cal C}(\{A_r\}_r,\{B_s\}_s):=\projlim_s\indlim_r\Hom_{\cal C}(A_r,B_s),\] where the right side is a genuine pro-ind limit in the category of sets, and composition is defined in the obvious way. For example, a pro object $\{A_r\}_r$ is isomorphic to zero (assuming that a zero object exists in $\cal C$, hence also in $\op{Pro}\cal C$) if and only if for each $r\ge 1$ there exists $s\ge r$ such that the transition map $A_s\to A_r$ is zero.

There is a fully faithful embedding $\cal C\to\op{Pro}\cal C$. Assuming $\cal C$ has countable inverse limits, this has a right adjoint $\op{Pro}\cal C\to \cal C,\; \{A_r\}_r\mapsto \projlim_r A_r$, which is left exact but not right exact. Moreover, if $\cal C$ is an abelian category then so is $\op{Pro}\cal C$: given a inverse system of exact sequences \[\cdots\To A_{n-1}(r)\To A_n(r)\To A_{n+1}(r)\To\cdots,\] the ``limit as $r\to\infty$'', namely \[\cdots \To \{A_{n-1}(r)\}_r\To \{A_n(r)\}_r\To \{A_{n+1}(r)\}_r\To\cdots,\] is an exact sequence in $\op{Pro}\cal C$.

Pro spectral sequences play an important role in the paper, which we will discuss for concreteness only in the case of abelian groups. Suppose that \[E^1_{pq}(r)\Longrightarrow H_{p+q}(r),\] for $r\ge 1$, are spectral sequences of abelian groups, which are functorial in that we have morphisms of spectral sequences $\cdots\to E^\bullet_{pq}(2)\to E^\bullet_{pq}(1)$. To avoid convergence issues, suppose that each spectral sequence is bounded, by a bound independent of $r$; e.g., each spectral sequence might be zero outside the first quadrant. Then we will often let ``$r\to\infty$'' to obtain a spectral sequence of pro abelian groups \[E^1_{pq}(\infty):=\{E^1_{pq}(r)\}_r\Longrightarrow \{H_{p+q}(r)\}_r.\] For further discussion and a dummy examples, see \cite[App.~A]{Morrow_pro_h_unitality}.

\subsection{Artin--Rees properties}
For the sake of reference, we now formally state a fundamental Artin--Rees result which will be used several times; this result appears to have been first noticed and exploited by M.~Andr\'e \cite[Prop.~10 \& Lem.~11]{Andre1974} and D.~Quillen \cite[Lem.~9.9]{Quillen1968}:

\begin{theorem}[Andr\'e, Quillen]\label{theorem_Artin_Rees}
Let $A$ be a Noetherian ring, and $I\subseteq A$ any ideal.
\begin{enumerate}
\item If $M$ is a finitely generated $A$-module, then the pro $A$-module $\{\Tor_n^A(A/I^r,M)\}_r$ vanishes for all $n>0$.
\item The ``completion'' functor
\begin{align*}
-\otimes_A A/I^\infty\,:A\op-{mod}&\To  \op{Pro}A\op{-mod}\\
M&\mapsto \{M\otimes_AA/I^r\}_r
\end{align*}
is exact on the subcategory of finitely generated $A$-modules.
\end{enumerate}
\end{theorem}
\begin{proof}[Sketch of proof]
By picking a resolution $P_\bullet$ of $M$ by finitely generated projective $A$-modules and applying the classical Artin--Rees property to the pair $d(P_n)\subseteq P_{n-1}$, one sees that for each $r\ge1$ there exists $s\ge r$ such that the map \[\Tor_n^A(A/I^s,M)\to\Tor_n^A(A/I^r,M)\] is zero. This proves (i). (ii) is just a restatement of (i).
\end{proof}

Next we quote a similar Artin--Rees result for Andr\'e--Quillen homology from \cite{Morrow_pro_h_unitality}, a companion to this paper which treats pro excision problems; the basics of Andr\'e--Quillen homology will be recalled in Section \ref{AQ_for_schemes}:

\begin{theorem}\label{theorem_AR_properties_in_AQ_homology}
Let $k\to A$ be a homomorphism of Noetherian rings, $I\subseteq A$ an ideal, and $M$ an $A$-module. Then:
\begin{enumerate}
\item $\{D_n^i(A/I^r|A,M/I^rM)\}_r=0$ for all $n\ge0$, $i\ge1$.
\item The canonical map $\{D_n^i(A|k,M/I^rM)\}_r\To\{D_n^i(A/I^r|k,M/I^rM)\}_r$ is an isomorphism for all $n\ge 0$, $i\ge0$.
\end{enumerate}
\end{theorem}
\begin{proof}
This is \cite[Thm.~4.3]{Morrow_pro_h_unitality}, where the pro H-unitality hypothesis of the cited theorem is satisfied thanks to \cite[Thm.~0.3]{Morrow_pro_h_unitality}.
\end{proof}

\section{Formal function properties and pro descent for Andr\'e--Quillen homology}\label{section_pro_h_for_AQ}
The aim of this section is to prove that Andr\'e--Quillen homology of proper schemes satisfies certain formal function properties. These will be established in Section \ref{subsection_formal_function_for_AQ}, after preliminary material on Andr\'e--Quillen homology is presented in Section \ref{AQ_for_schemes}. From these formal function properties we show in Section \ref{subsection_descent_for_AQ} that Andr\'e--Quillen homology satisfies the desired pro Mayer--Vietoris property with respect to abstract blow-up squares: this is a key step in deducing the same for $K$-theory in characteristic zero.

\subsection{Andr\'e--Quillen homology for schemes}\label{AQ_for_schemes}
In this preliminary section we recall the fundamentals of Andr\'e--Quillen homology \cite{Andre1974, Quillen1970, Ronco1993}, though we expect that the reader already has some familiarity with it, before extending it from algebras to schemes in the standard way using hypercohomology, as was done for Hochschild and cyclic homology in \cite{Weibel1991, Weibel1996}.

Let $k\to A$ be a homomorphism of rings; let $P_\bullet\to A$ be a simplicial resolution of $A$ by free $k$-algebras, and set \[\bb L_{A|k}:= \Omega_{P_\bullet|k}^1\otimes_{P_\bullet}A.\] Thus $\bb L_{A|k}$ is a simplicial $A$-module which is free in each degree; it is called the {\em cotangent complex} of the $k$-algebra $A$. The cotangent complex up to homotopy depends only on $A$, since the free simplicial resolution $P_\bullet\to A$ is unique up to homotopy.

Given simplicial $A$-modules $M_\bullet$, $N_\bullet$, the tensor product and alternating powers are new simplicial $A$-modules defined degreewise: $(M_\bullet\otimes_AN_\bullet)_n=M_n\otimes_AN_n$ and $(\bigwedge_A^rM_\bullet)_n=\bigwedge_A^rM_n$.

In particular we set $\bb L_{A|k}^i:=\bigwedge_A^i\bb L_{A|k}$ for each $i\ge 1$. The {\em Andr\'e--Quillen homology} of the $k$-algebra $A$, with coefficients in any $A$-module $M$, is defined by \[D_n^i(A|k,M):=H_n(\bb L_{A|k}^i\otimes_A M),\] for $n\ge 0$, $i\ge 1$. When $M=A$ the notation is simplified to \[D_n^i(A|k):=D_n^i(A|k,A)=H_n(\bb L_{A|k}^i).\] When $i=1$ the superscript is often omitted, writing $D_n(A|k,M)=H_n(\bb L_{A|k}\otimes_A M)$ and $D_n(A|k)=H_n(\bb L_{A|k})$ instead. If $k\to A$ is essentially of finite type and $k$ is Noetherian, then $D_n^i(A|k,M)$ is a finitely generated $A$-module for all $n,i$ and for all finitely generated $A$-modules $M$.

To avoid any ambiguity, we also remark that the notation $D_n^i(A|k,M)$ is defined in the same way if $i\le 0$ or $n<0$. However, $D_n^i(A|k,M)=0$ if $n<0$ and \[D_n^0(A|k,M)=\begin{cases}M&n=0,\\0&\mbox{else,}\end{cases}\] since $\bb L_{A|k}^0\otimes_AM\simeq M$.

Now we extend the definition of Andr\'e--Quillen homology to schemes; since we restrict to Noetherian schemes of finite Krull dimension, we may follow a classical approach, as we explain in the next remark:

\begin{remark}[Hypercohomology of non-bounded below complexes]\label{remark_hypercohomology} Our main results concerning Andr\'e--Quillen, and later Hochschild and cyclic, homology of schemes concern Noetherian schemes or even varieties; therefore we will restrict throughout to Noetherian schemes of finite Krull dimension to simplify hypercohomology of complexes which are not necessarily bounded below.

The Zariski hypercohomology of a cochain complex of sheaves $M^\blob$ on a Noetherian scheme $X$ of finite Krull dimension is defined as follows \cite[Appendix C]{Milne1980}: there exists a quasi-isomorphism $M^\blob\quis\cal F^\blob$, where $\cal F^\blob$ is a cochain complex of acyclic sheaves on $X$, and one shows that \[\bb H^*(X,M^\blob):=H^*(\cal F^\blob(X))\] is independent of $\cal F$. We may use the Cartan--Eilenberg approach to construct $\cal F^\blob$: first pick functorial acyclic resolutions $M^p\quis \cal F^{p\blob}$ of length $\le\dim X$ for each $p\in\bb Z$ (e.g., the Godement resolution), and then set $\cal F^\blob=\op{Tot}F^{\blob\blob}$; note that at most $1+\dim X$ terms appear in each direct sum $\cal F^n=\bigoplus_{p+q=n}\cal F^{p\,q}$.

The familiar spectral sequences
\begin{align*}
E_1^{pq}=H^q(X,M^p)&\implies \bb H^{p+q}(X,M^\blob)\\
'E_2^{pq}=H^p(X,H^q(M^\blob))&\implies\bb H^{p+q}(X,M^\blob)
\end{align*}
remain valid and bounded. We will also require two variations on Deligne's spectral sequence for the hypercohomology of a bounded-below filtered complex \cite[1.4.5]{Deligne1971}, which are not hard to prove given the boundedness assumptions we impose:
\begin{enumerate}
\item Let $M^{\blob\blob}$ be a bounded (i.e., $M^{p\,q}$ is non-zero for only finitely many $p,q$ in each bounded region of the plane), double cochain complex of sheaves. Then there is a bounded spectral sequence \[E_2^{pq}=\bb H^p(X,H^q_\sub{v}(M^{\bullet\bullet}))\implies \bb H^{p+q}(X,\op{Tot}M^{\bullet\bullet}),\] where $H_\sub v$ indicates taking cohomology in the second index.
\item Let $M^\blob$ be a cochain complex of sheaves equipped with a descending filtration $FM^\blob$ which satisfies $F^0M^\blob=M^\blob$ and $F^pM^\blob=0$ for $p\gg0$. Then there is a bounded spectral sequence \[E_1^{pq}=\bb H^{p+q}(X,\op{gr}^qM^\blob)\implies \bb H^{p+q}(X,M^\blob).\]
\end{enumerate}
\end{remark}

Now we are prepared to define Andr\'e--Quillen homology of schemes. Let $k$ be a ring and fix $i\ge 0$. First notice that if $R$ is a $k$-algebra then is possible to choose a simplicial resolution $P_\bullet(R)\to R$ by free $k$-algebras in a way which is functorial in $R$; e.g., via the comonad associating to $R$ the free $k$-algebra generated by the set $R$, see \cite[E.g.\ 8.6.16]{Weibel1994}. So, given a scheme $X$ over $k$, let $\tilde P_\blob(X)$ denote the simplicial sheaf of $k$-algebras obtained by degree-wise sheafifying $U\mapsto P_\blob(\roi_X(U))$; also put $\bb L_{X|k}^i:=\Omega^i_{\tilde P_\blob(X)/k}\otimes_{\tilde P_\blob(X)}\roi_X$. Given a quasi-coherent $\roi_X$-module $M$, and assuming that $X$ is Noetherian of finite Krull dimension, define the Andr\'e--Quillen homology of $X$, relative to $k$, with coefficients in $M$ to be \[D_n^i(X|k,M):=\bb H^{-n}(X,\bb L_{X|k,\sub{neg}}^i\otimes_{\roi_X}M),\] where the subscript neg indicates replacing the simplicial sheaf $\bb L_{X|k}^i$ by its associated chain complex and then negating the indexing to get a cochain complex of sheaves.

We continue the standard notational abuses of omitting $i$ when it equals $1$ and omitting $M$ when it equals $\roi_X$. Note, from (i) and (ii) below, that $D_n^i(X|k)$ can be non-zero for $n$ both positive and negative.

The following consequences are either automatic from the construction, or follow from the previous remark on hypercohomology and mimicking the arguments for Hochschild homology given in \cite{Weibel1991, Weibel1996}, bearing in mind that Andr\'e--Quillen homology of rings behaves well under localisation:

\begin{description}
\item[$\pmb{i=0}$:] Since $\bb L_{X|k,\sub{neg}}^0\otimes_{\roi_X}M\simeq M$, we have $D_n^0(X|k,M)\cong H^{-n}_\sub{Zar}(X,M)$.
\item[Agreement on affines:] If $X=\Spec A$ is an affine $k$-scheme, then $D_n^i(X|k,M)=D_n^i(A|k,M)$.
\item[Mayer--Vietoris property:] If $U,V$ are an open cover of $X$, then there is a long exact, Mayer--Vietoris sequence \[\cdots\to D_n^i(X|k,M)\to D_n^i(U|k,M|_U)\oplus D_n^i(V|k,M|_V)\to D_n^i(U\cap V|k,M|_{U\cap V})\to \cdots\]
\item[Hypercohomology spectral sequence:] There is a bounded fourth quadrant spectral sequence \[E_2^{pq}=H^p_\sub{Zar}(X,\cal D_{-q}^i(X|k,M))\Longrightarrow D_{-p-q}^i(X|k,M),\] where $\cal D_{-q}^i(X|k,M)$ denotes the Zariski sheafification of $U\mapsto D_{-q}^i(U|k,M|_U)$. This shows that $D_n^i(X|k,M)=0$ for $n<-\dim X$.
\end{description}
Moreover,
\begin{description}
\item[Homology l.e.s.:] If $0\to M\to N\to P\to 0$ is a short exact sequence of quasi-coherent $\roi_X$-modules, then there is a long exact sequence of Andr\'e--Quillen homology \[\cdots\To D_n^i(X|k,M) \To D_n^i(X|k,N)\To D_n^i(X|k,P)\To\cdots\]
(Proof: The long exact hypercohomology sequence associated to a short exact sequences of complexes of sheaves.)
\item[Higher Jacobi--Zariski spectral sequence:] Let $k\to A$ be a homomorphism of commutative rings, let $X$ be a Noetherian scheme over $A$ of finite Krull dimension, let $M$ be a quasi-coherent $\roi_X$-module, and fix $i\ge 0$. Then there is a natural, bounded spectral sequence of $A$-modules $E^1_{pq}\Longrightarrow D_{p+q}^i(X|k,M)$ whose columns may be described as follows:
\begin{enumerate}
\item Suppose $p<-i$ or $p>0$. Then $E^1_{pq}=0$.
\item Suppose $-i\le p\le0$. Then the $p^\sub{th}$ column of the $E^1$-page is given by a bounded spectral sequence \[\cal E^2_{\al\beta}=D_\al^{-p}\big(A|k,D_\beta^{i+p}(X|A,M)\big)\Longrightarrow E_{p,\al+\beta-p}^1.\] (Note that $\cal E_{\al\beta}^2=0$ if $\al$ or $\beta$ is $<-\dim X$.)
\end{enumerate}
(Proof: In the case that $X$ is affine, this is due to C.~Kassel and A.~Sletsj\o e \cite{Kassel1992}; a proof may also be found in \cite[Prop.~3.6]{Morrow_pro_h_unitality}. Examining the proof in the affine case reveals that $\bb L_{X|k}^i$ has a filtration with graded pieces $\op{gr}^p\bb L_{X|k}^i\simeq\bb L_{A|k}^p\otimes_A\bb L_{X/A}^{i-p}$ for $p=0,\dots,i$. Tensoring by $M$, and applying the variations on Deligne's spectral sequence for a filtered complex given in the Remark \ref{remark_hypercohomology}, one easily obtains the desired spectral sequence.)
\end{description}

\begin{example}
Suppose that $X$ is smooth over $k$, or more generally that $X$ has an affine open cover by the spectra of geometrically regular $k$-algebras. Then $\bb L_{X|k}^i\simeq\Omega_{X|k}^i$, and so $D_n^i(X|k)\cong H^{-n}(X,\Omega_{X|k}^i)$ for all $n\in\bb Z$, $i\ge0$.
\end{example}

This completes our discussion of Andr\'e--Quillen homology for schemes.
\subsection{Formal function properties for AQ homology of schemes}\label{subsection_formal_function_for_AQ}
Now that the basic properties of Andr\'e--Quillen homology for schemes have been established, we show that it satisfies formal function properties akin to Grothendieck's formal functions theorem. We begin with a lemma in the affine case:

\begin{lemma}
Let $A\to B$ be an essentially finite type morphism of Noetherian rings, $I\subseteq B$ an ideal, and $M$ a finitely generated $B$-module. Then the canonical map \[\{D_n^i(B|A,M)\otimes_AA/I^r\}\To\{D_n^i(B|A,M/I^rM)\}_r\] is an isomorphism for all $n,i\ge0$.
\end{lemma}
\begin{proof}
Given another $B$-module $N$, there is a natural, first quadrant spectral sequence \[E_{pq}^2=\Tor_p^B(D_q^i(B|A,M),N)\Longrightarrow D_{p+q}^i(B|A,M\otimes_BN).\] This is proved in the usual way by choosing a resolution of $N$ by flat $B$-modules and applying the Kunneth spectral sequence for a tensor product of complexes.

Applying this with $N=B/I^r$ gives spectral sequences \[E_{pq}^2(r)=\Tor_p^B(D_q^i(B|A,M),B/I^r)\Longrightarrow D_{p+q}^i(B|A,M/I^rM),\] and taking the limit over $r$ yields a first quadrant spectral sequence of pro $B$-modules: \[E^2_{pq}(\infty)=\{\Tor_p^B(D_q^i(B|A,M),B/I^r)\}_r\Longrightarrow \{D_{p+q}^i(B|A,M/I^rM)\}_r.\] Our hypotheses on $A,B$ and $M$ ensure that the $B$-modules $D_q^i(B|A,M)$ are finitely generated, whence Theorem \ref{theorem_Artin_Rees}(i) implies that $E^2_{pq}(\infty)=0$ for $p>0$. We thus obtain edge map isomorphisms $\{D_n^i(B|A,M)\otimes_BB/I^r\}\isoto\{D_n^i(B|A,M/I^rM)\}$, as desired.
\end{proof}

In the following results a sheaf of ideals $\cal I\subseteq\roi_X$ plays a prominent role; to keep notation clear, we write $rY$ for the $r^\sub{th}$-infinitesimal thickening of the closed subscheme defined by $\cal I$, i.e., \[rY:=\Spec\roi_X/\cal I^r.\]

\begin{theorem}\label{theorem_formal_function_for_AQ}
Let $A$ be a Noetherian ring and $\pi:X\to\Spec A$ a Noetherian scheme over $A$ of finite Krull dimension; let $\cal I\subseteq\roi_X$ be an ideal sheaf and $M$ a coherent $\roi_X$-module. Then:
\begin{enumerate}
\item The canonical map $\{D_n^i(X|A,M/\cal I^rM)\}_r\To\{D_n^i(rY|A,M/\cal I^rM)\}_r$ is an isomorphism for all $n\in\bb Z$, $i\ge0$.
\item Assuming that $\pi$ is essentially of finite type, the canonical map of pro sheaves $\{\cal D_n^i(X|A,M)\otimes_{\roi_X}\roi_X/\cal I^r\}_r\To\{\cal D_n^i(X|A,M/\cal I^rM)\}_r$ is an isomorphism for all $n\in\bb Z$, $i\ge0$.
\end{enumerate}
Now suppose further that $\pi$ is proper and that $\cal I=I\roi_X$ for some ideal $I\subseteq A$. Then:
\begin{enumerate}\setcounter{enumi}{2}
\item $D_n^i(X|A,M)$ is a finitely generated $A$-module for all $n\in\bb Z$, $i\ge 0$.
\item The canonical map $\{D_n^i(X|A,M)\otimes_AA/I^r\}_r\To\{D_n^i(X|A,M/\cal I^rM)\}_r$ is an isomorphism for all $n\in\bb Z$, $i\ge0$.
\end{enumerate}
\end{theorem}
\begin{proof}
(i): By the Mayer--Vietoris long exact descent sequence, and the usual two inductions starting with an affine open cover of $X$ (note that $X$ is quasi-separated, since it is Noetherian), this claim reduces to the case that $X$ is affine, in which case it is Theorem \ref{theorem_AR_properties_in_AQ_homology}(ii).

(ii): Assume $\pi$ is essentially of finite type. This claim also reduces to the affine case, which is the previous lemma, but since this is the only occurrence of pro sheaves, we provide a few more details. Fixing $n\in\bb Z$ and $i\ge0$, let $K_r$ and $C_r$ respectively denote the kernel and cokernel of the canonical map \[\cal D_n^i(X|A,M)\otimes_{\roi_X}\roi_X/\cal I^r\To\cal D_n^i(X|A,M/\cal I^rM).\] Let $\{U_i\}$ be a finite cover of $X$ by affine opens. Then, given $r\ge1$ and applying the previous lemma, there exists $s\ge r$ such that the maps $K_s(U_i)\to K_r(U_i)$ and $C_s(U_i)\to C_r(U_i)$ are zero; since these are quasi-coherent sheaves, it follows that the maps $K_s\to K_r$ and $C_s\to C_r$ are also zero, whence $\{K_r\}_r$ and $\{C_r\}_r$ vanish, completing the proof.

For the rest of the proof we assume that $\pi$ is proper and that $\cal I=I\roi_X$ for some ideal $I\subseteq A$.

(iii): Each $\roi_X$-module $\cal D_n^i(X|A,M)$ is coherent since $X$ is of finite type over $A$. Since $X$ is proper over $A$, this implies that each of the cohomology groups of $\cal D_n^i(X|A,M)$ is a finitely generated $A$-module; so the claim follows from the hypercohomology spectral sequence.

(iv): As explained in the proof of (iii), the hypercohomology spectral sequence \[E_2^{pq}=H^p(X,\cal D_{-q}^i(X|A,M))\Longrightarrow D_{-p-q}^i(X|A,M)\] consists of finitely generated $A$-modules. Since $\{-\otimes_AA/I^r\}_r$ is an exact functor on the category of finitely generated $A$-modules, by Theorem \ref{theorem_Artin_Rees}(ii), we may apply it to the spectral sequence to obtain a spectral sequence of pro $A$-modules \[E_2^{pq}(\infty)=\{H^p(X,\cal D_{-q}^i(X|A,M))\otimes_AA/I^r\}_r\Longrightarrow \{D_{-p-q}^i(X|A,M)\otimes_AA/I^r\}_r.\] This spectral sequence maps to the limit of the hypercohomology spectral sequences for the $\roi_X$-modules $M/\cal I^rM$, namely \['E_2^{pq}(\infty)=\{H^p(X,\cal D_{-q}^i(X|A,M/\cal I^rM))\}_r\Longrightarrow \{D_{-p-q}^i(X|A,M/\cal I^rM)\}_r,\] and so to complete the proof it is enough to show that the canonical maps $E_2^{pq}(\infty)\to {'E_2^{pq}}(\infty)$ are all isomorphisms.

Firstly, since $\cal D_{-q}^i(X|A,M)$ is a coherent sheaf and $\pi$ is proper, Grothendieck's formal functions theorem\footnote{
This theorem is normally stated as the isomorphism $\projlim_r\,H^p(X,N)\otimes_AA/I^r\isoto\projlim_r\,H^p(X,N/\cal I^rN)$ for any coherent $\roi_X$-module $N$. But a quick examination of the proof in EGA shows that the stronger isomorphism of pro $A$-modules $\{H^p(X,N)\otimes_AA/I^r\}_r\isoto\{H^p(X,N/\cal I^rN)\}_r$ also holds.}\cite[Cor.~4.1.7]{EGA_III_I}
states that the canonical map \[E_2^{pq}(\infty)=\{H^p(X,\cal D_{-q}^i(X|A,M))\otimes_AA/I^r\}_r\To \{H^p(X,\cal D_{-q}^i(X|A,M)\otimes_{\roi_X}\roi_X/\cal I^r)\}_r\] is an isomorphism. Secondly, part (ii) implies that the canonical map \[\{H^p(X,\cal D_{-q}^i(X|A,M)\otimes_{\roi_X}\roi_X/\cal I^r)\}_r\To  \{H^p(X,\cal D_{-q}^i(X|A,M/\cal I^rM))\}_r= {'E_2^{pq}}(\infty)\] is also an isomorphism, completing the proof.
\end{proof}

It is part (iv) of the previous theorem which will be so essential in Section \ref{subsection_descent_for_AQ}; a secondary application of it is the following corollaries, which will not be needed but which most closely mimic Grothendieck's formal function theorem for coherent cohomology.

\begin{corollary}\label{corollary_traditional_formal_functions_for_AQ}
Let $A$ be a Noetherian ring, $X$ a proper scheme over $A$ of finite Krull dimension, $I\subseteq A$ an ideal, and $M$ a coherent $\roi_X$-module. Then the canonical map \[\{D_n^i(X|A,M)\otimes_AA/I^r\}_r\To \{D_n^i(rY|A/I^r,M/I^rM)\}_r\] is an isomorphism for all $n\in\bb Z$, $i\ge0$, where $rY=X\times_AA/I^r$.
\end{corollary}
\begin{proof}
Using parts (iv) and (i) of the previous theorem, it remains only to prove that the canonical map $\{D_n^i(rY|A,M/I^rM)\}_r\to \{D_n^i(rY|A/I^r,M/I^rM)\}_r$ is an isomorphism. From the higher Jacobi--Zariski spectral sequence, this follows from the vanishing of $\{D_n^i(A/I^r|A)\}_r$; the details are as follows:

Applying the higher Jacobi--Zariski spectral sequences to the homomorphism $A\to A/I^r$ and $A/I^r$-scheme $rY$, we obtain a bounded spectral sequence of pro $A$-modules $E^1_{pq}(\infty)\Longrightarrow \{D_{p+q}^i(rY|A,M/I^rM)\}_r$ which is supported in the range $-i\le p\le 0$, where it is described by bounded, right-half-plane spectral sequences \[\cal E_{\al\beta}^2(\infty)=\{D_\al^{-p}(A/I^r|A,D_\beta^{i+p}(rY|A/I^r,M/I^rM))\}_r\Longrightarrow E^1_{p,\al+\beta-p}(\infty).\]

For any $r,\al>0$, Theorem \ref{theorem_AR_properties_in_AQ_homology}(i) may be applied to the ideal $I^r$ and module $D_\beta^{i+p}(rY|A/I^r,M/I^rM)$ to deduce that there exists $s>1$ such that the second of the following maps, hence the composition, is zero:
\begin{align*}
D_\al^{-p}(A/I^{sr}|A,D_\beta^{i+p}(srY|A/I^{sr},M/I^{sr}M))
&\To D_\al^{-p}(A/I^s|A,D_\beta^{i+p}(rY|A/I^r,M/I^rM))\\
&\To D_\al^{-p}(A/I^r|A,D_\beta^{i+p}(rY|A/I^r,M/I^rM))
\end{align*}
That is, $\cal E_{\al\beta}^2(\infty)=0$ for $\al\neq0$. Also, $\cal E_{0\beta}^2(\infty)=0$ for $p\neq0$. So the $\cal E$-spectral sequences degenerate to edge map isomorphisms
\[E^1_{p,n-p}(\infty)\cong \begin{cases}0&p\neq0\\D^i_n(rY|A/I^r,M/I^rM)&p=0\end{cases}\] Thus the $E$-spectral sequence degenerates to the desired edge map isomorphisms.
\end{proof}

\begin{corollary}[Formal functions theorem for Andr\'e--Quillen homology]
Under the hypotheses of the previous corollary, the canonical map \[D_n^i(X|A,M)^\comp\To\projlim_rD_n^i(rY|A/I^r,M/I^rM)\] is an isomorphism for all $n\in\bb Z$, $i\ge0$, where $D_n^i(X|A,M)^\comp$ denotes the $I$-adic completion of $D_n^i(X|A,M)$.
\end{corollary}
\begin{proof}
Take $\projlim_r$ of the pro isomorphism of the previous corollary.
\end{proof}

\begin{remark}
Let $A$ be a Noetherian ring, $X$ a proper scheme over $A$ of finite Krull dimension, and $I\subseteq A$ an ideal with respect to which $A$ is assumed to be adically complete. Applying the previous corollary to $M=\roi_X$ and noting that $D_n^i(X|A)$ is a finitely generated $A$-module, hence already $I$-adically complete, we obtain \[D_n^i(X|A)\Isoto\projlim_rD_n^i(rY|A/I^r).\]
\end{remark}

\subsection{Pro descent for AQ homology with respect to blow-up squares}\label{subsection_descent_for_AQ}
With the formal function theorems for Andr\'e--Quillen homology established, we may now proved that Andr\'e--Quillen homology satisfies the desired pro Mayer--Vietoris property with respect to abstract blow-up squares. We begin by developing further the results of Theorem \ref{theorem_formal_function_for_AQ} in the case that $\pi$ is a an isomorphism away from $V(I)$:

\begin{lemma}\label{lemma_AQ_of_birational}
Let $k\to A$ be a homomorphism of Noetherian rings, $I\subseteq A$ an ideal, and $X$ a proper scheme over $A$ of finite Krull dimension such that the induced morphism $X\setminus V(I\roi_X)\to\Spec A\setminus V(I)$ is an isomorphism. Then the following canonical maps are isomorphisms for all $n\in\bb Z$, $i\ge0$:
\begin{enumerate}
\item $\{D_n^i(X|A,I^r\roi_X)\}_r\To\{I^rD_n^i(X|A)\}_r\stackrel{(\ast)}{\cong}0$ \\(vanishing $(\ast)$ not necessarily valid if $i=n=0$).
\item $\{D_n^i(A|k,I^r)\}\To\{D_n^i(X|k,I^r\roi_X)\}_r$.
\end{enumerate}
\end{lemma}
\begin{proof}
(i): The short exact sequences $0\to I^r\roi_X\to \roi_X\to \roi_X/I^r\roi_X\to 0$ induce a long exact sequence of pro $A$-modules \[\cdots\To \{D_n^i(X|A,I^r\roi_X)\}_r\To D_n^i(X|A)\To\{D_n^i(X|A,\roi_X/I\roi_X)\}_r\To\cdots\] According to Theorem \ref{theorem_formal_function_for_AQ}(iv), $\{D_n^i(X|A,\roi_X/I\roi_X)\}_r\cong\{D_n^i(X|A)\otimes_AA/I^r\}_r$, from which it follows that the long exact sequences breaks into short sequences and induces the desired isomorphisms $\{D_n^i(X|A,I^r\roi_X)\}_r\isoto\{I^rD_n^i(X|A)\}_r$.

Regarding the vanishing claim, suppose first that $i=0$; then $D_n^0(X|A)=H^{-n}(X,\roi_X)$, which is a finitely generated $A$-module (since $X$ is proper over $A$) which is supported on $V(I)$ when $n\neq0$ (since $X\setminus V(I\roi_X)\cong\Spec A\setminus V(I)$), and hence is killed by a power of $I$. Next suppose $i>0$; in this case the vanishing claim merely requires that $X$ be essentially of finite type over $A$; induction on the size of an affine open cover of $X$ reduces us to proving that if $B$ is an essentially finite type $A$-algebra then $D_n^i(B|A)$ is killed by a power of $I$. But $D_n^i(B|A)$ is a finitely generated $B$-module supported on $V(IB)$, since Andr\'e--Quillen homology behaves well under localisation; thus it is killed by a power of $I$, completing the proof of all the vanishing claims.

(ii): The higher Jacobi--Zariski spectral sequence states that for each $r\ge 1$ there is a natural, third quadrant, bounded, spectral sequence \[E_{pq}^1(r)\Longrightarrow D_{p+q}^i(X|k,I^r\roi_X)\] which vanishes outside $-i\le p\le 0$, and whose columns in the range $-i\le p\le 0$ are described by natural, first quadrant spectral sequences \[\cal E_{\al\beta}^2(r)=D_\al^{-p}(A|k,D_\beta^{i+p}(X|A,I^r\roi_X))\Longrightarrow E_{p,\al+\beta-p}^1(r).\] According to part (i), $\{\cal E_{\al\beta}^2(r)\}_r=0$ unless $\beta=0$ and $p=-i$, whence the limit of the spectral sequences collapse to edge map isomorphisms \[\{D_n^i(A|k,D_0^0(X|A,I^r\roi_X))\}_r\isoto\{D_n^i(X|k,I^r\roi_X)\}_r.\]

Since $\{D_0^0(X|A,I^r\roi_X)\}_r\cong \{I^rD_0^0(X|A)\}_r=\{I^rH^0(X,\roi_X)\}_r$ by part (i) and the usual description of Andr\'e--Quillen homology when $i=0$, we will have completed the proof as soon as we show that the canonical map \[\{I^r\}_r\To\{I^rH^0(X,\roi_X)\}_r\] is an isomorphism of pro $A$-modules. Write $B:=H^0(X,\roi_X)$, which is a finite $A$-algebra with the property that $A_\frak p \isoto B\otimes_AA_\frak p$ for each prime ideal $\frak p\subseteq A$ not containing $I$; hence the kernel $K$ and cokernel $C$ of the structure map $A\to B$ are killed by a power of $I$. Now consider the following commutative diagram of pro $A$-modules:
\[\xymatrix{
0 \ar[r] & K \ar[r]\ar[d]^{(1)} & A \ar[r]\ar[d]^{(2)} & B \ar[r]\ar[d]^{(3)} & C \ar[r] \ar[d]^{(4)}& 0 \\
0 \ar[r] & \{K\otimes_AA/I^r\}_r \ar[r] & \{A/I^r\}_r \ar[r] & \{B\otimes_AA/I^r\}_r \ar[r] & \{C\otimes_AA/I^r\}_r \ar[r] & 0
}\]
The top row is an exact sequence of finitely generated $A$-modules, whence the bottom row is an exact sequence of pro $A$-modules by Theorem \ref{theorem_Artin_Rees}(ii). Since $K$ and $C$ are killed by a power of $I$, maps (1) and (4) are isomorphisms, and so the central square in the diagram is bicartesian. It follows that maps (2) and (3) have isomorphic kernels; i.e., $\{I^r\}_r\isoto\{I^rB\}_r$, as required.
\end{proof}

\begin{corollary}\label{corollary_AQ_of_birational}
Let $k$ be a Noetherian ring, let $X'\to X$ be a proper morphism between Noetherian $k$-schemes of finite Krull dimension, and suppose that $\cal I\subseteq\roi_X$ is an ideal sheaf such that the induced morphism $X'\setminus V(\cal I\roi_{X'})\to X\setminus V(\cal I)$ is an isomorphism. Then the canonical map \[\{D_n^i(X|k,\cal I^r)\}_r\To\{D_n^i(X'|k,\cal I^r\roi_{X'})\}_r\] is an isomorphism for all $n\in\bb Z$, $i\ge0$.
\end{corollary}
\begin{proof}
By the usual two inductions starting with an affine open cover of $X$ (since $X$ is Noetherian, hence quasi-separated) we may assume that $X=\Spec A$ is affine, in which case the statement is precisely part (ii) of the previous lemma (with $X$ rechristened $X'$).
\end{proof}

We are now equipped to prove our main theorem about Andr\'e--Quillen homology of schemes, namely that it satisfies the pro Mayer--Vietoris property with respect to abstract blow-up squares. Recall that an {\em abstract blow-up square} of schemes is a pull-back diagram \xysquare{Y'}{X'}{Y}{X}{->}{->}{->}{->} where $X'\to X$ is proper, $Y\to X$ is a closed embedding, and the induced map on the open subschemes $X'\setminus Y'\to X\setminus Y$ is an isomorphism.

To state our descent result we need the following piece of notation which we did not introduce in Remark \ref{remark_hypercohomology} when discussing hypercohomology: if $M^\blob$ is a cochain complex of sheaves, and $M^\blob\quis\cal F^\blob$ is a quasi-isomorphism to a cochain complex of acyclic sheaves, then set \[\bb H(X,M^\blob):=\cal F^\blob(X),\] which is well-defined up to quasi-isomorphism and satisfies $H^*(\bb H(X,M^\blob))=\bb H^*(X,M^\blob)$ by definition.

The following theorem is true more generally with a coherent module $M$ in place of $\roi_X$, but we omit it to simplify the notation:

\begin{theorem}[Pro descent for AQ homology wrt.~abstract blow-ups]\label{theorem_pro_cdh_for_AQ}
Let $k$ be a Noetherian ring, and let
\[\xymatrix{
Y'\ar[d]\ar[r] & X'\ar[d]\\
Y\ar[r] & X
}\]
be an abstract blow-up square of Noetherian, finite Krull dimensional $k$-schemes. Then the following square of cochain complexes of pro $k$-modules is homotopy cartesian
\xysquare{\bb H(X,\bb L_{X|k,\sub{\rm neg}}^i)}{\bb H(X',\bb L_{X'|k,\sub{\rm neg}}^i)}{\{\bb H(rY,\bb L_{rY|k,\sub{\rm neg}}^i)\}_r}{\{\bb H(rY',\bb L_{rY'|k,\sub{\rm neg}}^i)\}_r}{->}{->}{->}{->}
resulting in a long exact, Mayer--Vietoris sequence of pro $k$-modules \[\cdots\To D_n^i(X|k)\To\{D_n^i(rY|k)\}_r\oplus\{D_n^i(X'|k)\}_r\To D_n^i\{(rY'|k)\}_r\To\cdots\]
\end{theorem}
\begin{proof}
According to Theorem \ref{theorem_formal_function_for_AQ}(i), the cohomology of the complexes in the bottom row of the diagram are unchanged if we replace them by \[\{\bb H(X,\bb L^i_{X|k,\sub{neg}}\otimes_{\roi_X}\roi_X/\cal I^r)\}_r\To\{\bb H(X',\bb L^i_{X'|k,\sub{neg}}\otimes_{\roi_{X'}}\roi_{X'}/\cal I^r\roi_{X'})\}_r,\] where $\cal I$ is the sheaf of ideals defining $Y$. Having made this replacement, we must show that the homotopy fibres of the two resulting vertical arrows
\xysquare{\bb H(X,\bb L_{X|k,\sub{neg}}^i)}{\bb H(X',\bb L_{X'|k,\sub{neg}}^i)}{\{\bb H(X,\bb L^i_{X|k,\sub{neg}}\otimes_{\roi_X}\roi_X/\cal I^r)\}_r}{\{\bb H(X',\bb L^i_{X'|k,\sub{neg}}\otimes_{\roi_{X'}}\roi_{X'}/\cal I^r\roi_{X'}\}_r)}{}{->}{->}{}
are quasi-isomorphic. These homotopy fibres are respectively \[\{\bb H(X,\bb L^i_{X|k,\sub{neg}}\otimes_{\roi_X}\cal I^r)\}_r\quad\mbox{and}\quad\{\bb H(X',\bb L^i_{X'|k,\sub{neg}}\otimes_{\roi_{X'}}\cal I^r\roi_{X'})\}_r,\] between which the canonical map is indeed a quasi-isomorphism since $\{D_n^i(X|k,\cal I^r)\}_r\isoto\{D_n^i(X'|k,\cal I^r\roi_{X'})\}_r$ for all $n\in\bb Z$, $i\ge0$, by Corollary \ref{corollary_AQ_of_birational}.
\end{proof}

\section{Formal function properties and pro descent for $\HH$, $\HC$, $\THH$, $\TC$, and $K$-theory}\label{section_pro_h_for_K}
Now we extend the main results of Section \ref{section_pro_h_for_AQ} to Hochschild and cyclic homology of schemes, and their topological counterparts, and then deduce that $K$-theory also satisfies the pro Mayer--Vietoris property with respect to abstract blow-up squares of varieties.

\subsection{Hochschild and cyclic homology}\label{subsection_HH_and_HC_of_schemes}
Here we discuss the theories of derived Hochschild and cyclic homology for schemes which will concern us.

Let $k\to A$ be a homomorphism of commutative rings. Given an $A$-module $M$, we let  $H_*^{\sub{naive},k}(A,M)$  denote the ``usual'' Hochschild homology of $A$ as a $k$-algebra with coefficients in $M$; it is the homology of the Hochschild complex $C_\bullet^k(A,M)$. In particular, $\HH_*^{\sub{naive},k}(A)=H_*^{\sub{naive},k}(A,A)$ denotes the usual Hochschild homology of $A$ as a $k$-algebra. However, we will work throughout with the derived version of Hochschild homology, for which we use the notation $H_*^k(A,M)$. That is, letting $P_\bullet\to A$ be a simplicial resolution of $A$ by free $k$-algebras, $H_*^k(A,M)$ is defined to be the homology (of the diagonal) of the bisimplicial $A$-module \[C_q^k(P_p,M)=M\otimes_kP_p^{\otimes_k q}.\tag{$p,q\ge 0$}\] In the special case $A=M$ we write $\HH_*^k(A)=H_*^k(A,A)$.

Next we discuss cyclic homology. Firstly, $\HC_*^{\sub{naive},k}(A)$ denotes the usual cyclic homology of the $k$-algebra $A$, which is the homology of the cyclic bicomplex $CC_{\bullet\bullet}^k(A)$. Just as for Hochschild homology we prefer to denote by $\HC_*^k(A)$ the derived version, defined as the homology of (the diagonal of) the simplicial tricomplex $CC_{\bullet\bullet}^k(P_\bullet)$, where $P_\bullet\to A$ is a simplicial resolution of $A$ by free $k$-algebras. The usual SBI sequence remains valid in the derived setting: \[\cdots\To \HH_n^k(A)\stackrel{I}{\To} \HC_n^k(A)\stackrel{S}{\To} \HC_{n-2}^k(A)\stackrel{B}{\To}\cdots\]

Next consider a Noetherian scheme $X$ of finite Krull dimension over $k$. As in Section \ref{AQ_for_schemes}, let $\tilde P_\bullet^k(X)$ denote the simplicial sheaf of $k$-algebras obtained by degree-wise sheafifying $U\mapsto P_\bullet(\roi_X(U))$, where $P_\bullet(\roi_X(U))\to\roi_X(U)$ is a functorially chosen simplicial resolution by free $k$-algebras. Given a quasi-coherent $\roi_X$-module $M$, set \[C_\bullet^k(X,M):=M\otimes_k\tilde P_\bullet^k(X)^{\otimes_k\bullet}\] and let $C_\sub{neg}^k(X,M)$ be the cochain complex of quasi-coherent $\roi_X$-modules, supported in negative degrees, obtained by negating the numbering of the chain complex of $C_\bullet^k(X,M)$. Define the derived Hochschild homology of $X$, relative to $k$, with coefficients in $M$ to be its hypercohomology \[H_*^k(X,M):=\bb H^{-*}(X,C_\sub{neg}^k(X,M)).\] In the special case $M=\roi_X$ we write $C_\blob^k(X)=C_\blob^k(X,\roi_X)$ and $\HH_*^k(X)=H_*^k(X,\roi_X)$. The obvious analogues of the ``Agreement on affines'', ``Mayer--Vietoris property'', ``Hypercohomology spectral sequence'', and ``Homology l.e.s.'' stated for Andr\'e--Quillen homology in Section \ref{AQ_for_schemes} are true for this derived Hochschild homology of schemes.

We omit the analogous construction of derived cyclic homology for schemes since the procedure should now be clear. As in the affine case, the usual SBI sequence remains valid \[\cdots\To \HH_n^k(X)\stackrel{I}{\To} \HC_n^k(X)\stackrel{S}{\To} \HC_{n-2}^k(X)\stackrel{B}{\To}\cdots\]

The relationships between Andr\'e--Quillen homology, derived Hochschild/cyclic homology, and usual Hochschild/cyclic homology are summarised by the next lemma:

\begin{lemma}
Let $k$ be a ring, $X$ a Noetherian scheme of finite Krull dimension over $k$, and $M$ a quasi-coherent $\roi_X$-module. Then:
\begin{enumerate}
\item If $X$ is flat over $k$ then the canonical maps \[H_n^k(X,M)\to H_n^{\sub{naive},k}(X,M)\quad\mbox{and}\quad  \HC_n^k(X)\to \HC_n^{\sub{naive},k}(X)\] are isomorphisms for all $n\in\bb Z$, where $H^{\sub{naive},k}$ and $\HC^{\sub{naive},k}$ denote the ``usual'' Hochschild and cyclic homology of a scheme, as constructed in, e.g., \cite{Weibel1996}.
\item There is a bounded spectral sequence of $k$-modules \[E^2_{pq}=D_p^q(X|k,M)\Longrightarrow H_{p+q}^k(X,M).\]
\end{enumerate}
\end{lemma}
\begin{proof}
These are relatively standard results; the affine case may be found in \cite[Lem.~4.1]{Morrow_pro_h_unitality}, and the proofs for a scheme are then clear.
\end{proof}

The following formal function results for Hochschild homology are the analogues of those given for Andr\'e--Quillen homology in Section \ref{section_pro_h_for_AQ}.

\begin{theorem}\label{theorem_formal_function_for_HH}
Let $k\to A$ be a homomorphism of Noetherian rings, and $\pi:X\to\Spec A$ a Noetherian scheme over $A$ of finite Krull dimension; let $\cal I\subseteq\roi_X$ be an ideal sheaf and $M$ a coherent $\roi_X$-module. Then:
\begin{enumerate}
\item The canonical map $\{H_n^A(X,M/\cal I^rM)\}_r\To\{H_n^A(rY,M/\cal I^rM)\}_r$ is an isomorphism for all $n\in\bb Z$.
\item Assuming that $\pi$ is essentially of finite type, the canonical map of pro sheaves $\{\cal H_n^A(X,M)\otimes_{\roi_X}\roi_X/\cal I^r\}_r\To\{\cal H_n^A(X,M/\cal I^rM)\}_r$ is an isomorphism for all $n\in\bb Z$.
\end{enumerate}
Now suppose further that $\pi$ is proper and $\cal I=I\roi_X$ for some ideal $I\subseteq A$. Then
\begin{enumerate}\setcounter{enumi}{2}
\item $H_n^A(X,M)$ is a finitely generated $A$-module for all $n\in\bb Z$,
\end{enumerate}
and the following canonical maps are isomorphisms for all $n\in\bb Z$:
\begin{enumerate}\setcounter{enumi}{3}
\item $\{H_n^A(X,M)\otimes_AA/I^r\}_r\To\{H_n^A(X,M/\cal I^rM)\}_r$.
\item $\{H_n^A(X,M)\otimes_AA/I^r\}_r\To\{H_n^{A/I^r}(rY,M/\cal I^rM)\}_r$.
\end{enumerate}
Now suppose still further that the induced morphism $X\setminus V(I\roi_X)\to\Spec A\setminus V(I)$ is an isomorphism. Then the following canonical maps are isomorphisms for all $n\in\bb Z$:
\begin{enumerate}\setcounter{enumi}{5}
\item $\{H_n^A(X,I^r\roi_X)\}_r\To\{I^r\HH_n^A(X)\}_r\stackrel{(\ast)}{\cong}0$\\(vanishing $(\ast)$ not necessarily valid if $n=~0$).
\item $\{H_n^k(A,I^r)\}_r\To\{H_n^k(X,I^r\roi_X))\}_r$.
\end{enumerate}
\end{theorem}
\begin{proof}
Thanks to the Andr\'e--Quillen to Hochschild homology spectral sequence for schemes (part (ii) of the previous lemma), these claims immediately reduce to the analogous assertions for Andr\'e--Quillen homology, which we have already proved. For the sake of reference: (i)--(iv) correspond to Theorem \ref{theorem_formal_function_for_AQ}; (v) to Corollary \ref{corollary_traditional_formal_functions_for_AQ}; and (vi)--(vii) to Lemma \ref{lemma_AQ_of_birational}.
\end{proof}

Given a closed embedding $Y\into X$ of Noetherian schemes of finite Krull dimension, we will write $\HH_*^k(X,Y)$, $\HC_*^k(X,Y)$ for the relative Hochschild and cyclic homology groups, fitting into long exact sequences
\begin{align*}
\cdots\To \HH_n^k(X,Y)\To \HH_n^k(X)\To \HH_n^k(Y)\To\cdots\\
\cdots\To \HC_n^k(X,Y)\To \HC_n^k(X)\To \HC_n^k(Y)\To\cdots
\end{align*}
Now we can prove the analogue of Theorem \ref{theorem_pro_cdh_for_AQ} for $\HH$ and $\HC$, namely that they also satisfy the pro Mayer--Vietoris property for abstract blow-up squares:

\begin{theorem}[Pro descent for $\HH$ and $HC$ wrt.~abstract blow-ups]\label{theorem_pro_cdh_for_HH}
Let $k$ be a Noetherian ring, and let
\[\xymatrix{
Y'\ar[d]\ar[r] & X'\ar[d]\\
Y\ar[r] & X
}\]
be an abstract blow-up square of Noetherian, finite Krull dimensional $k$-schemes. Then the canonical maps \[\{\HH_n^k(X,rY)\}_r\To \{\HH_n^k(X',rY')\}_r,\quad\quad \{\HC_n^k(X,rY)\}_r\To \{\HC_n^k(X',rY')\}_r\] are isomorphisms of pro abelian groups for all $n\in\bb Z$.
\end{theorem}
\begin{proof}
The claim for $\HH$ is equivalent to the statement that the following square of cochain complexes of pro $k$-modules is homotopy cartesian, where $\bb H$ is the hypercohomology replacement functor immediately proceeding Theorem \ref{theorem_pro_cdh_for_AQ}:
\xysquare{\bb H(X,C^k_\sub{neg}(X))}{\bb H(X',C^k_\sub{neg}(X'))}{\{\bb H(rY,C^k_\sub{neg}(rY))\}_r}{\{\bb H(rY',C^k_\sub{neg}(rY'))\}_r}{->}{->}{->}{->} This is proved exactly as it was for Andr\'e--Quillen homology in Theorem \ref{theorem_pro_cdh_for_AQ}, using the appropriate parts of the previous theorem. More precisely, by part (i) of the previous theorem we may replace the bottom of the diagram by \[\{\bb H(X,C^k_\sub{neg}(X,\roi_X/\cal I^r))\}_r\To\{\bb H(X,C^k_\sub{neg}(X',\roi_{X'}/\cal I^r\roi_{X'}))\}_r\] without changing the cohomology of the complexes, where $\cal I$ is the sheaf of ideals defining $Y$. Then the map between the homotopy fibres of the vertical arrows becomes \[\{\bb H(X,C^k_\sub{neg}(X,\cal I^r))\}_r\To\{\bb H(X,C^k_\sub{neg}(X',\cal I^r\roi_{X'}))\}_r,\] which is a quasi-isomorphism by part (vii) of the previous theorem (to be precise, by the $\HH$ analogue of Corollary \ref{corollary_AQ_of_birational}, which follows from part (vii) of the previous lemma by induction on the size of an affine open cover of $X$). This proves the claim for $\HH$.

To pass to $\HC$, use the five lemma and induction up the limit of the SBI sequences
\[\xymatrix{
\cdots\ar[r] & \HH_n^k(X,rY)\ar[r]\ar[d] & \HC_n^k(X,rY)\ar[r]\ar[d] & \HC_{n-2}^k(X,rY)\ar[r]\ar[d] & \cdots \\
\cdots\ar[r] & \HH_n^k(X',rY')\ar[r] & \HC_n^k(X',rY')\ar[r] & \HC_{n-2}^k(X,rY)\ar[r] & \cdots
}\]
(notice it is possible to start the induction since the SBI sequences vanishes in degrees $<-\max\{\dim X,\dim X'\}$).
\end{proof}

\subsection{Topological Hochschild and cyclic homology}\label{subsection_HH_and_HC_of_schemes}
The analogues for topological Hochschild and cyclic homology of the formal functions properties given in Theorem \ref{theorem_formal_function_for_HH} were recently proved in joint work with B.~Dundas \cite{Morrow_Dundas}; from this we will quickly prove that these theories also satisfy the pro Mayer--Vietoris property with respect to abstract blow-up squares.

We will assume that the reader is familiar with topological Hochschild and cyclic homology, and so only offer a brief summary. For a quasi-compact, quasi-separated scheme $X$, the spectra $\THH(X)$, $\TR^m(X;p)$, and $\TC^m(X;p)$ were defined by Geisser and Hesselholt in \cite{GeisserHesselholt1999} in such a way that all these presheaves of non-connective spectra satisfy Zariski descent (see the proof of \cite[Corol.~3.3.3]{GeisserHesselholt1999}). They were shown in \cite{Morrow_Dundas} to have reasonable finiteness and continuity properties under the assumption of F-finiteness, where we fix a prime number $p$:

\begin{definition}\label{definition_F-finite}
A $\bb Z_{(p)}$-algebra is said to be {\em F-finite} if and only if the $\bb F_p$-algebra $A/pA$ is finitely generated over its subring of $p^\sub{th}$-powers. A $\bb Z_{(p)}$-scheme is said to be F-finite if and only if it admits a finite open cover by the spectra of F-finite $\bb Z_{(p)}$-algebras.
\end{definition}

In particular, given a Noetherian, F-finite $\bb Z_{(p)}$-algebra $A$, the ring of truncated $p$-typical Witt vectors $W_m(A)$ is again a Noetherian, F-finite $\bb Z_{(p)}$-algebra; see \cite[\S2]{Morrow_Dundas} for a detailed discussion of such issues. Recall that $\TR^m_n(A;p)$ is naturally a $W_m(A)$-algebra.

The following establishes that topological Hochschild homology and the related theories all have the pro Mayer--Vietoris property for abstract blow-up squares:

\begin{theorem}[Pro descent for $\THH$, $\TR^m$, $\TC^m$ wrt. abstract blow-ups]\label{theorem_pro_cdh_for_THH}
Let
\[\xymatrix{
Y'\ar[d]\ar[r] & X'\ar[d]\\
Y\ar[r] & X
}\]
be an abstract blow-up square of Noetherian, F-finite, finite Krull dimensional $\bb Z_{(p)}$-schemes. Then the canonical maps
\begin{align*}
\{\THH_n(X,rY;\bb Z/p^v)\}_r&\To \{\THH_n(X',rY';\bb Z/p^v)\}_r\\
\{\TR_n^m(X,rY;p,\bb Z/p^v)\}_r&\To \{\TR_n^m(X',rY';p,\bb Z/p^v)\}_r\\
\{\TC_n^m(X,rY;p,\bb Z/p^v)\}_r&\To \{\TC_n^m(X',rY';p,\bb Z/p^v)\}_r
\end{align*}
are isomorphisms of pro abelian groups for all $n\in\bb Z$ and $m,v\ge1$.
\end{theorem}
\begin{proof}
Since $\THH=\TR^1$ and since $\TC^m=\holim(\TR^m\xto{1-F}\TR^{m-1})$, it is sufficient to prove the claim for $\TR^m$. By Zariski descent we may assume that $X=\Spec A$ is affine.

That is, $A$ is a Noetherian, $F$-finite $\bb Z_{(p)}$-algebra, $I\subseteq A$ is an ideal, $\pi:X'\to \Spec A$ is a proper morphism which induces an isomorphism $X'\setminus V(I\roi_X)\to\Spec A\setminus V(I)$, and we must prove that the square of pro spectra
\[\xymatrix{
\TR^m(A;p,\bb Z/p^v) \ar[r]\ar[d] & \{\TR^m(A/I^r;p,\bb Z/p^v)\}_r\ar[d]\\
\TR^m(X';p,\bb Z/p^v) \ar[r] & \{\TR^m(X'\times_AA/I^r;p,\bb Z/p^v)\}_r
}\tag{\dag}\]
is homotopy cartesian. By \cite[Thm.~5.6]{Morrow_Dundas}, the hypotheses imply that the canonical maps
\begin{align*}
\{\TR_n^m(A;p,\bb Z/p^v)\otimes_{W_m(A)}W_m(A/I^r)\}_r&\To\{\TR_n^m(A/I^r;p,\bb Z/p^v)\}_r\\
\{\TR_n^m(X;p,\bb Z/p^v)\otimes_{W_m(A)}W_m(A/I^r)\}_r&\To\{\TR_n^m(X\times_AA/I^r;p,\bb Z/p^v)\}_r
\end{align*}
are isomorphisms of pro $W_m(A)$-modules for all $n\in\bb Z$. Hence the horizontal arrows in (\dag) induce surjections at the level of the homotopy groups, and so it is sufficient to show that the induced vertical map on the kernels, namely \[\{W_m(I^r)\,\TR_n^m(A;p,\bb Z/p^v)\}_r\To \{W_m(I^r)\,\TR_n^m(X;p,\bb Z/p^v)\}_r,\] is an isomorphism for all $n\in\bb Z$. The rest of the proof is devoted to proving this isomorphism.

Firstly, since $\TR_n^m$ behaves well under localisation (see Lem.~3.1 of \cite{Morrow_Dundas} and its proof), the canonical map \[\TR_n^m(X;p,\bb Z/p^v)\otimes_{W_m(A)}W_m(A_f)\To \TR_n^m(X\times_AA_f;p,\bb Z/p^v)\] is an isomorphism for any $f\in A$. In particular, if $f\in I$ so that $X\times_AA_f\cong\Spec A_f$, then the canonical map \[\TR_n^m(A;p,\bb Z/p^v)\otimes_{W_m(A)}W_m(A_f)\To \TR_n^m(X;p,\bb Z/p^v)\otimes_{W_m(A)}W_m(A_f)\] is an isomorphism. Since the rest of the proof is commutative algebra, write $M:=\TR_n^m(A;p,\bb Z/p^v)$ and $N:=\TR_n^m(X;p,\bb Z/p^v)$; these are finitely generated as $W_m(A)$-modules, by \cite[Thm.~5.3]{Morrow_Dundas}, and we have just proved that the map $M\to N$ becomes an isomorphism after applying $-\otimes_{W_m(A)}W_m(A_f)$ for any $f\in I$. 

Secondly, we claim that the spectra $\Spec W_m(A_f)$, $f\in I$, form an open affine cover of $\Spec W_m(A)\setminus V(W_m(I))$. This follows from the following two facts:
\begin{align*}
W_m(A_f)&=W_m(A)_{[f]},\\
W_m(I)^M&\subseteq\pid{\mbox{ideal of $W_m(A)$ generated by }[f]:\;f\in I}\mbox{ for }M\gg0,
\end{align*}
where $[f]\in W_m(A)$ denotes the Teichm\"uller lift of an element $f\in A$.
Both these facts are relatively standard results about Witt rings, e.g., see \cite[Lem.~A.6(i)]{Rulling2007} and \cite[Lem.~2.1 \& Rem.~2.2]{Morrow_Dundas} respectively.

It now follows from commutative algebra that the kernel and cokernel of the map $M\to N$ are killed by a power of $W_m(I)$. So, by the same argument as used at the end of the proof of Lemma \ref{lemma_AQ_of_birational}, the induced map $\{W_m(I)^rM\}_r\to\{W_m(I)^rN\}_r$ is an isomorphism. Finally, to complete the proof, note that the chains of ideals $W_m(I)^r$ and $W_m(I^r)$ are intertwined, by \cite[Lem.~2.1]{Morrow_Dundas}.
\end{proof}

\subsection{Pro cdh-descent for $K$-theory}\label{subsection_pro_cdh_K}
We now prove the main theorem of the paper, namely that $K$-theory satisfies the pro Mayer--Vietoris property with respect to abstract blow-up squares of varieties.

In the interest of not being restricted to varieties over a field, we prefer to work in the generality of Noetherian, quasi-excellent schemes whenever possible.  Recall that a Noetherian scheme $X$ is called quasi-excellent if and only if the formal fibre of each point of $X$ is geometrically regular and the regular locus of any finite type $X$-scheme is open. All ``naturally occurring'' schemes in algebraic geometry are quasi-excellent, and it is intimately related to resolution of singularities: according to Grothendieck \cite[Prop.~7.9.5]{EGA_IV_II}, if $X$ is a Noetherian scheme with the property that every integral, finite type $X$-scheme admits a desingularisation, then $X$ is quasi-excellent. If $X$ is furthermore assumed to be a $\bb Q$-scheme, then the converse is true by H.~Hironaka \cite{Hironaka1964} and M.~Temkin \cite[Thm.~2.3.6]{Temkin2008}.

The following refinement of the Haesemeyer argument is the key to understanding Noetherian, quasi-excellent $\bb Q$-schemes; I am grateful to Haesemeyer for discussions about this:

\begin{proposition}\label{proposition_Haesemeyer}
Let $\cal E$ be a presheaf of spectra on the category of Noetherian, quasi-excellent $\bb Q$-schemes with the following properties: $\cal E$ is invariant under nilpotent extensions, satisfies Nisnevich descent, and satisfies the Mayer--Vietoris property for blow-ups along regularly embedded centres. Then $\cal E$ satisfies the Mayer--Vietoris property for all abstract blow-up squares.
\end{proposition}
\begin{proof}
To simplify the explanation, we begin by indicating an abstract blow-up square of Noetherian, quasi-excellent, $\bb Q$-schemes and its value under $\cal E$:
\[(1)\quad\xymatrix{
Y' \ar[r]\ar[d] & X' \ar[d]\\
Y\ar[r]& X
}
\qquad\qquad
(2)\quad \xymatrix{
\cal E(X) \ar[r]\ar[d] & \cal E(X') \ar[d]\\
\cal E(Y)\ar[r]& \cal E(Y')
}
\]
According to the Haesemeyer argument \cite[\S5--6]{Haesemeyer2004}, axiomised into the form in which we are using it in \cite[Thm.~3.12]{Cortinas2008}, our assumptions apply that square (2) is homotopy cartesian as soon as square (1) consists of finite type $k$ schemes for some characteristic zero field $k$. However, examination of Haesmeyer's original argument reveals that the only required properties of $k$ are that it be infinite and have strong resolution of singularities in the sense specified in \cite[Thm.~2.4]{Haesemeyer2004}. But Hironaka \cite{Hironaka1964} proved strong resolution of singularities over any Noetherian, quasi-excellent local ring $A$ with characteristic zero residue field; so Haesemeyer's argument works in this generality, and we deduce that square (2) is homotopy cartesian whenever square (1) consists of finite type $A$-schemes, for such a ring $A$.

Finally, let square (1) consist of arbitrary Noetherian, quasi-excellent $\bb Q$-schemes. Since $\cal E$ satisfies Zariski descent, square (2) is homotopy cartesian as soon as the squares
\[\xymatrix{
\cal E(\Spec\roi_{X,x}) \ar[r]\ar[d] & \cal E(X'\times_X\Spec\roi_{X,x}) \ar[d]\\
\cal E(Y\times_X\Spec\roi_{X,x})\ar[r]& \cal E(Y'\times_X\Spec\roi_{X,x})
}\]
are homotopy cartesian for all points $x\in X$. But since $A=\roi_{X,x}$ is a Noetherian, quasi-excellent local ring with characteristic residue field, these square are indeed homotopy cartesian, by the previous paragraph.
\end{proof}
 
We have reached the main theorem of the paper: 
 
\begin{theorem}[Pro descent for $K$-theory wrt.~abstract blow-ups]\label{theorem_pro_descent_for_K_theory}
Let \xysquare{Y'}{X'}{Y}{X}{->}{->}{->}{->} be an abstract blow-up square of either
\begin{enumerate}
\item Noetherian, quasi-excellent $\bb Q$-schemes of finite Krull dimension; or
\item finite type schemes over an infinite, perfect field $k$ which has strong resolution of singularities.
\end{enumerate}
Then the canonical map \[\{K_n(X,rY)\}_r\To\{K_n(X',rY')\}_r\] of pro abelian relative $K$-groups is an isomorphism for all $n\in\bb Z$.
\end{theorem}
\begin{proof}
(i): For any quasi-compact, quasi-separated $\bb Q$-scheme $Z$, let $H\!N^{\bb Q}(Z)$ and $H\!P^{\bb Q}(Z)$ denote its negative and periodic cyclic homologies over $\bb Q$, as formulated in \cite{Cortinas2008}. The infinitesimal $K$-theory of $Z$ is defined to be the homotopy fibre of the Chern character from $K$-theory to negative cyclic homology, i.e., \[K^\sub{inf}(Z):=\op{hofib}(K(Z)\xto{ch} N\!H^{\bb Q}(Z)).\] According to [Thm.\ 4.6, loc.~cit.] and [Cor.~3.13, loc.~cit.], or rather their improvements using Proposition \ref{proposition_Haesemeyer}, the presheaves of spectra $K^\sub{inf}$ and $H\!P$ carry any abstract blow-up square of Noetherian, quasi-excellent $\bb Q$-schemes to a homotopy cartesian square of spectra; so the canonical maps of relative spectra \[K^\sub{inf}(X,rY)\To K^\sub{inf}(X',rY'),\quad\quad H\!P^{\bb Q}(X,rY)\To H\!P^{\bb Q}(X',rY')\] are weak-equivalences for all $r\ge1$.

From the long exact sequence $\cdots\to H\!N_n^{\bb Q}\to H\!P_n^{\bb Q}\to \HC_{n-2}^{\bb Q}\to\cdots$, Theorem \ref{theorem_pro_cdh_for_HH} with $k=\bb Q$, and the usual five lemma argument, we may now deduce that the canonical map $\{H\!N_n^{\bb Q}(X,rY)\}_r\to \{H\!N_n^{\bb Q}(X',rY')\}_r$ is an isomorphism for all $n\in\bb Z$. Applying the same argument to the long exact sequence $\cdots\to K^\sub{inf}_n\to K_n\to N\!H_n\to\cdots$ we deduce that $\{K_n(X,rY)\}_r\to\{K_n(X',rY')\}_r$ is an isomorphism for all $n\in\bb Z$, as desired.

(ii): The infinitesimal $K$-theory $\{K^{\sub{inf},m}(Z)\}_m$ of a finite type $k$-scheme $Z$ is defined to be the pro spectrum arising as the homotopy fibres of the trace map from $K$-theory to topological cyclic homology, i.e., \[K^{\sub{inf},m}(Z):=\op{hofib}(K(Z)\xto{tr}\TC^m(Z;p)).\] According to \cite[Thm.~B]{GeisserHesselholt2010} (see also \cite{Krishna2009a}) the square of pro spectra \xysquare{\{K^{\sub{inf},m}(X)\}_m}{\{K^{\sub{inf},m}(Y)\}_m}{\{K^{\sub{inf},m}(X')\}_m}{\{K^{\sub{inf},m}(Y')\}_m}{->}{->}{->}{->} is homotopy cartesian. Therefore our claim for $K$-theory follows, in a similar way to characteristic zero, from the analogous assertion for topological cyclic homology, namely Theorem \ref{theorem_pro_cdh_for_HH} (note that the hypotheses of Theorem \ref{theorem_pro_cdh_for_HH} are satisfied for finite type $k$-schemes, and that for such schemes the finite coefficients can be eliminated everywhere; see \cite[Corol.~5.9]{Morrow_Dundas} for related discussion).
\end{proof}

\begin{remark}
There are three special cases in which Theorem \ref{theorem_pro_descent_for_K_theory} is already known or can be otherwise deduced:
\begin{enumerate}
\item If $Y\to X$ is a regular immersion and $X'$ is the blow-up of $X$ along $Y$, then it follows from R.~Thomason's blow-up formula \cite{Thomason1993} that there are short exact sequences of pro abelian groups \[0\To K_n(X)\To\{K_n(rY)\}_r\oplus K_n(X')\To\{K_n(rY')\}_r\To0,\] which is a stronger statement than the conclusion of Theorem \ref{theorem_pro_descent_for_K_theory}.
\item If $X'\to X$ is a finite morphism, then Theorem \ref{theorem6} reduces to pro excision, which was established for arbitrary commutative, Noetherian rings in \cite{Morrow_pro_h_unitality}.
\item If $X$ is a Cohen--Macaulay variety over an infinite field with only isolated singularities, $Y=X_\sub{sing}$ and $X'\to X$ is a resolution of singularities, then a technique of Weibel \cite{Weibel2001} using minimal reduction ideals implies that $X'$ may be obtained by first blowing up $X$ along a regular immersion and then normalising. So in this case Theorem \ref{theorem_pro_descent_for_K_theory} reduces to combining cases (i) and (ii). This case of desingularising a Cohen--Macaulay variety with isolated singularities first proved by Krishna \cite[Proof of Thm.~1.2]{Krishna2010} in characteristic zero.
\end{enumerate}
The validity of these special cases of Theorem \ref{theorem_pro_descent_for_K_theory} was our motivation for establishing it in full generality. It is a commonly encountered issue in cdh-descent problems that typical resolutions of singularities cannot be factored into regularly embedded blow-ups and normalisations, e.g.~\cite[e.g.~6.3]{Haesemeyer2004}. Since the Haesemeyer argument of Proposition \ref{proposition_Haesemeyer} is not valid here (the pro spectra depend depend not only on the given schemes, but also on the embeddings), there appears to be no way to assemble the aforementioned cases (i)--(iii) into a proof of Theorem \ref{theorem_pro_descent_for_K_theory}.
\end{remark}

\section{Applications of Theorem \ref{theorem_pro_descent_for_K_theory}}\label{section_application_1}
As mentioned in the introduction, the main applications of Theorem \ref{theorem_pro_descent_for_K_theory}, which are to zero cycles on singular varieties, are presented in an accompanying paper \cite{Morrow_zero_cycles}. Here we offer some more straightforward applications to demonstrate its use, first by reformulating it as a definition of $K$-theory of compact support and then using it to give a quick proof of parts of Weibel's $K$-dimension conjecture on negative $K$-theory.

For simplicity we adopt the following conventions in this section:
\begin{quote}
A field $k$ will be called {\em good} if and only if it is infinite, perfect, and has strong resolution of singularities, e.g., $\op{char}k=0$ suffices.

A {\em $k$-variety} means simply a finite type $k$-scheme; further assumptions will be specified when required, and the reference to $k$ with occasionally be omitted.
\end{quote}

\subsection{$K$-theory with compact support}\label{subsection_compact_support}
Here we explain how Theorem \ref{theorem_pro_descent_for_K_theory} may be interpreted as the well-definedness of ``$K$-theory with compact support''; I am grateful to H.~Gillet for suggesting this interpretation to me. In this section $k$ is a good field.

By Nagata, any separated $k$-variety $X$ has a compactification $\overbar X$, i.e.~$\overbar X$ is proper over $k$ and $X$ is a dense open subset of $\overbar X$; we refer the reader to \cite{Lutkebohmert1993} for the proof.

\begin{definition}
Let $X$ be a separated $k$-variety. Let $\overbar X$ be any compactification of $X$, and set $Y=\overbar X\setminus X$, equipped with any structure as a closed subscheme. The {\em $K$-theory of $X$ with compact support} is defined to be the spectrum \[K^c(X):=\op{holim}_rK(\overbar X,rY),\] and the $K$-groups of $X$ with compact support are defined to be its homotopy groups.
\end{definition}

The content of the following proposition is equivalent to taking homotopy limits in Theorem \ref{theorem_pro_descent_for_K_theory} in the case of separated $k$-varieties:

\begin{proposition}\label{proposition_compact1}
The previous definition does not depend on the chosen compactification $\overbar X$ of $X$.
\end{proposition}
\begin{proof}
Suppose that $\res X_i$, $i=1,2$ are different compacitifications of $X$; then, by \cite[Corol.~2.4]{Lutkebohmert1993}, there exists a third compacitification $\overbar X_3$ which is a blow-up of both $\overbar X_1$ and $\overbar X_2$. Let $Y_1,Y_2,Y_3$ be the complements of $X$ in the compactifications. Then diagram \xysquare{Y_3}{X_3}{Y_i}{X_i}{->}{->}{->}{->} is an abstract blow-up square for $i=1,2$, and so Theorem \ref{theorem_pro_descent_for_K_theory} implies that the maps \[\{K_n(\overbar X_1,rY_1)\}_r\To\{K_n(\overbar X_3,rY_3)\}_r\longleftarrow \{K_n(\overbar X_2,rY_2)\}_r\] are isomorphisms for all $n\in\bb Z$. Hence the maps
\[\op{holim}_rK(\overbar X_1,rY_1)\To\op{holim}_rK(\overbar X_3,rY_3)\longleftarrow \op{holim}_rK(\overbar X_2,rY_2)\] are weak equivalences, as required.
\end{proof}

\begin{remark}
If $\overbar X$ is a proper variety containing $X$ as an open, but not-necessarily dense, subscheme, then \xysquare{Y':=X'\setminus X}{X'}{Y:=\overbar X\setminus X}{\overbar X}{->}{->}{->}{->} is an abstract blow-up square, where $X'$ denotes the closure of $X$ inside $\overbar X$. So Theorem \ref{theorem_pro_descent_for_K_theory} implies that $K^c(X)\simeq\op{holim}_rK(\overbar X,rY)$. In other words, $K^c(X)$ may be defined with respect to any proper variety containing $X$ as an open subset: it is not necessary that $X$ be dense. This will be implicitly used in the next result.
\end{remark}

To justify its definition as a theory with compact support, we show that $K^c$ fits into a localisation sequence:

\begin{proposition}\label{proposition_compact2}
Let $Y\to X$ be a closed embedding of separated $k$-varieties, with open complement $U\subseteq X$. Then there is a functorial homotopy fibre sequence \[K^c(U)\To K^c(X)\To\op{holim}_rK^c(rY).\]
\end{proposition}
\begin{proof}
The proof is a chase of the definitions in which some care is required regarding infinitesimal thickenings.

Firstly, we may suppose $Y$ is reduced since that does not affect its system of thickenings in $X$. Let $\overbar X$ be a compactification of $X$, and set $Z:=\overbar X\setminus X$ and $W:=\overbar X\setminus U$, with their reduced subscheme structure to obtain closed embeddings $Z\to W\to\overbar X$. Then there is a resulting commutative diagram of homotopy fibre sequences
\[\xymatrix{
K(\overbar X,r_{\overbar X}Z) \ar[r]\ar[d] & K(\overbar X)\ar[r]\ar@{=}[d] & K(r_{\overbar X}Z)\ar[d]\\
K(\overbar X,r_{\overbar X}W) \ar[r] & K(\overbar X)\ar[r] & K(r_{\overbar X}W)
}\]
in which the notations $r_{\overbar X}Z=rZ$ and $r_{\overbar X}W=rW$ are to remind the reader and author that the infintesimal thickenings are taken inside $\overbar X$. Taking $\op{holim}_r$ and homotopy cofibres of the vertical maps yields a homotopy fibre sequence \[K^c(U)\To K^c(X)\To \op{holim}_rK(r_{\overbar X}W,r_{\overbar X}Z).\]

Next we claim that the open complement of the closed embedding $r_{\overbar X}Z\to r_{\overbar X}W$ is precisely $r_XY$, where this latter infinitesimal thickening is taken inside $X$; this follows from the identifications \[r_{\overbar X}Z\times_{\overbar X}X=r_X(Z\times_{\overbar X}X)=r_XY.\] Hence $K^c(r_XY)=\op{holim}_sK^c(r_{\overbar X}W, s(r_{\overbar X}Z))$, where we are taking the $s^\sub{th}$ infinitesimal thickening of $r_{\overbar X}Z$ inside $r_{\overbar X}W$. To complete the proof it remains to check that the canonical map \[\op{holim}_{r,s}K(r_{\overbar X}W, s(r_{\overbar X}Z))\To \op{holim}_rK(r_{\overbar X}W,r_{\overbar X}Z)\] is a weak equivalence. This is true simply because the inverse systems on each side are intertwined. (This is easiest to see in the affine case, when $\overbar X=\Spec A$, $W=\Spec A/J$, and $Z=\Spec A/I$ with $I\supseteq J$; then the canonical map of interest can be written as $\op{holim}_{r,s}K(A/J^r,(J^r+I^{sr})/J^r)\to\op{holim}_r K(A/J^r,I^r/J^r)$, in which the two systems are intertwined because the transition map $K(A/J^{sr},I^{sr}/J^{sr})\to K(A/J^r,I^r/J^r)$ factors through $K(A/J^r,(J^r+I^{sr})/J^r)$.)
\end{proof}

\begin{remark}
We finish this section with two remarks about this $K$-theory with compact support.
\begin{enumerate}
\item More generally, suppose that $\pi:X\to S$ is a separated morphism of $k$-varieties. Picking a proper morphism $\overbar X\to S$ of $k$-varieties compactifiying $\pi$ and setting $Y:=\overbar X\setminus X$, we may define the {\em $K$-theory with compact support of the family $\pi$} as \[K^c(X/S):=\op{holim}_rK(\overbar X,rY).\] Verbatim repeating the proofs of Propositions \ref{proposition_compact1} and \ref{proposition_compact2} we see that this is well-defined and satisfies localisation.
\item The construction in (i) remains valid for any separated morphism of Noetherian, quasi-excellent $\bb Q$-schemes of finite Krull dimension. For example, if $\op{char}k=0$ and $X$ is a smooth, proper variety over $k((t))$, then \[K^c(X/k[[t]]):=\op{holim}_rK(\overbar X,rY)\] is a well-defined invariant of $X$, where $\overline X\to\Spec k[[t]]$ is any proper model of $X$, with special fibre denoted $Y:=\overline X\times_{k[[t]]}k$. It seems likely that $K^c_0(X/k[[t]])$ is related to nearby cycles.
\end{enumerate}
\end{remark}

\subsection{Negative $K$-groups}
The following conjecture concerning negative $K$-theory was raised by Weibel in 1980 \cite[Qu.~2.9]{Weibel1980}:
\begin{quote}
{\bf $K$-dimension conjecture:} If $X$ is a Noetherian scheme of dimension $d$, then $K_n(X)=0$ for $n<d$, and $X$ is $K_{-d}$-regular.
\end{quote}
The conjecture was proved in \cite{Cortinas2008} for varieties over characteristic zero fields, where it was also shown that $K_{-d}$ of any such scheme is a finitely generated group (this is not explicitly stated, but is a straightforward consequence of \cite[Thm.~0.2(1)]{Cortinas2008}), which had been proved earlier for mild types of singularities in \cite[\S8]{Haesemeyer2004}.

The vanishing part of the conjecture was proved in \cite{GeisserHesselholt2010} and \cite{Krishna2009a} for varieties over good fields of finite characteristic.

Our aim in this section is to show that Theorem \ref{theorem_pro_descent_for_K_theory} offers a very quick proof of the vanishing part of the $K$-dimension conjecture, and of the finite generation of $K_{-d}$, for the following types of schemes:
\begin{enumerate}
\item Noetherian, quasi-excellent $\bb Q$-schemes; and
\item varieties over good fields of finite characteristic.
\end{enumerate}
We do not consider the $K_{-d}$-regularity part of the conjecture.

The following standard lemma reduces the problems to reduced schemes:

\begin{lemma}\label{lemma_reduction_to_reduced}
If $X$ is a Noetherian scheme of dimension $d$, then the canonical map $K_n(X)\to K_n(X_\sub{red})$ is an isomorphism for $n\le d$.
\end{lemma}
\begin{proof}
Given a nilpotent ideal $I$ of a ring $R$, the canonical map $K_n(R)\to K_n(R/I)$ is an isomorphism for $n=0$, hence for $n<0$. So the canonical map from the descent spectral sequence for $X$, namely \[E_2^{pq}=H^p(X,\cal K_{-q,X})\implies K_{-p-q}(X),\] to that for $X_\sub{red}$, \[_\sub{red}E_2^{pq}=H^p(X_\sub{red},\cal K_{-q,X_\sub{red}})\implies K_{-p-q}(X),\] is an isomorphism on the $E_2$ page whenever $q\ge0$. The claim easily follows from this.
\end{proof}

\begin{theorem}\label{theorem_K_dim}
Let $X$ be a $d$-dimensional scheme which is either
\begin{enumerate}
\item a Noetherian, quasi-excellent $\bb Q$-scheme; or
\item a variety over a good field.
\end{enumerate}
Then $K_n(X)=0$ for $n<-d$, and the group $K_{-d}(X)$ is finitely generated.
\end{theorem}
\begin{proof}
The proof is by induction on $d$; for any fixed $d$ it is enough to treat the case that $X$ is reduced, thanks to Lemma \ref{lemma_reduction_to_reduced}.

Firstly, if $X$ is zero-dimensional and reduced then it is a disjoint union of the spectra of fields. So all the negative $K$-groups vanish, and $K_0(X)$ is the free abelian group generated by the finitely many points of $X$.

The inductive step proceeds by picking a desingularisation $X'\to X$; in case (ii) this exists by assumption, while in case (i) it exists by Temkin's extension \cite[Thm.~2.3.6]{Temkin2008} of Hironaka's resolution theorem. That is, there exists an abstract blow-up square \xysquare{Y'}{X'}{Y}{X}{->}{->}{->}{->} in which $X'$ is regular, and $Y'$ and $Y$ have dimension $<d$. By Theorem \ref{theorem_pro_descent_for_K_theory} there is a resulting long exact, Mayer--Vietoris sequence of pro abelian groups \[\cdots\To K_n(X)\To \{K_n(rY)\}_r\oplus K_n(X')\To\{K_n(rY')\}_r\To\cdots\]  But $K_n(X')=0$ for all $n<0$ since $X'$ is regular, and $K_n(rY)=K_n(Y)$ and $K_n(rY')=K_n(Y')$ for all $n\le -(d-1)$ by Lemma \ref{lemma_reduction_to_reduced}; hence the Mayer--Vietoris sequence simplifies in degrees $\le d$ to an exact sequence \[K_{-(d-1)}(Y')\to K_{-d}(X)\to K_{-d}(Y)\to K_{-d}(Y')\to K_{-d-1}(X)\to K_{-d-1}(Y)\to \cdots\]
But, by the inductive hypothesis, $K_n(Y)$ and $K_n(Y')$ vanish for $n<-(d-1)$, and $K_{-(d-1)}(Y')$ is finitely generated. This evidently completes the proof.
\end{proof}

\small
\bibliographystyle{acm}
\bibliography{../Bibliography}

\noindent Matthew Morrow\hfill {\tt morrow@math.uni-bonn.de}\\
Mathematisches Institut\hfill \url{http://www.math.uni-bonn.de/people/morrow/}\\\
Universit\"at Bonn\\
Endenicher Allee 60\\
53115 Bonn, Germany
\end{document}